\documentclass[11pt]{article}
\usepackage[utf8]{inputenc}
\usepackage[english]{babel}
\usepackage{cancel}
\usepackage{csquotes}
\usepackage{graphicx}
\usepackage{yhmath}
\usepackage{color}
\usepackage[utf8]{inputenc}
\usepackage{appendix}
\usepackage{amsmath}
\usepackage{amsthm}
\usepackage{amsfonts}
\usepackage{amssymb}
\usepackage{longtable}
\usepackage{cite}
\usepackage{url}
\usepackage{graphics}
\usepackage{subcaption}
\usepackage{epsfig}
\usepackage{geometry}
\usepackage{verbatim}
\usepackage{tikz}
\usetikzlibrary{fit,shapes.misc}

\DeclareMathOperator{\GL}{GL}

\usepackage{listings,xcolor}
\definecolor{dkgreen}{rgb}{0,.6,0}
\definecolor{dkblue}{rgb}{0,0,.6}
\definecolor{dkyellow}{cmyk}{0,0,.8,.3}

\newtheorem{theorem}{Theorem}[section]

\newtheorem{remark}[theorem]{Remark}
\newtheorem{cor}[theorem]{Corollary}
\newtheorem{prop}[theorem]{Proposition}
\newtheorem{definition}[theorem]{Definition}

\newtheorem{oss}[theorem]{Observation}

\newtheorem{ex}[theorem]{Example}
\def\field{\mathbb{}}

\title{New scattered subspaces in higher dimensions}
\author{Daniele Bartoli\thanks{Department of Mathematics and Informatics, University of Perugia, Perugia, Italy.
Email address: daniele.bartoli@unipg.it}\hspace{0.15 cm}, Alessandro Giannoni\thanks{Department of Mathematics and applications ``Renato Caccioppoli", University Federico II of Napoli, Napoli, Italy.
Email address: alessandro.giannoni@unina
.it}\hspace{0.15 cm} and Giuseppe Marino\thanks{Department of Mathematics and applications ``Renato Caccioppoli", University Federico II of Napoli, Napoli, Italy.
Email address: giuseppe.marino@unina
.it}\hspace{0.15 cm}}

\date{}

\begin{document}
\maketitle

\makeatother
\newcommand{\Prf}{\noindent{\bf Proof}.\quad }
\renewcommand{\labelenumi}{(\alph{enumi})}

%6. Contents of the headers
\def\as{{^{\sigma}}}
\def\e{{\`{e }}}
\def\a{{\`{a}}}
\def\qi{{^{q^I}}}
\def\qii{{^{q^{2I}}}}
\def\qj{{^{q^J}}}
\def\qjj{{^{q^{2J}}}}
\def\qio{{^{q^{I_0}}}}
\def\qjo{{^{q^{J_0}}}}
\def\qk{{^{q^K}}}
\def\qji{{^{q^{J-I}}}}
\def\qij{{^{q^{I+J}}}}
\def\qjia{{^{q^{J-I}+1}}}
\def\qija{{^{q^{I-J}+1}}}
\def\field{{\mathbb F_{q^n}}}
\def\us{U_{\mathbf{A}}^{I,J}}
\def\fax{{x^{q^I} +\alpha y^{q^J}}}
\def\fbx{{x\qj +\beta z\qi}}
\def\fcx{{y\qi +\gamma z\qj}}
\def\faxl{{\lambda\qi x\qi +\alpha\lambda\qj y\qj}}
\def\fbxl{{\lambda\qj x\qj +\beta\lambda\qi z\qi}}
\def\fcxl{{\lambda\qi y\qi +\gamma\lambda\qj z\qj}}
\newcommand{\fal}[2]{{\lambda\qi #1\qi +\alpha\lambda\qj #2\qj}}
\newcommand{\fbl}[2]{{\lambda\qj #1\qj +\beta\lambda\qi #2\qi}}
\newcommand{\fcl}[2]{{\lambda\qi #1\qi +\gamma\lambda\qj #2\qj}}
\newcommand{\fa}[2]{{#1^{q^I} +\alpha #2^{q^J}}}
\newcommand{\fb}[2]{{#1\qj +\beta #2\qi}}
\newcommand{\fc}[2]{{#1\qi +\gamma #2\qj}}

\begin{abstract}
Over the past few decades, there has been extensive research on scattered subspaces, partly because of their link to MRD codes. These subspaces can be characterized using linearized polynomials over finite fields. Within this context, scattered sequences extend the concept of scattered polynomials and can be viewed as geometric equivalents of exceptional MRD codes. Up to now, only scattered sequences of orders one and two have been developed. However, this paper presents an infinite series of exceptional scattered sequences of any order beyond two which correspond to scattered subspaces that cannot be obtained as direct sum of scattered subspaces in smaller dimensions. The paper also addresses equivalence concerns within this framework.
\end{abstract}

\bigskip


\section{Introduction}

% Let $q$ be a prime power, $n$ be a positive integer, and denote by $\mathbb{F}_{q^n}$ the finite field with $q^n$ elements. \ale{Questa la toglierei dall intro}
Scattered subspaces 
%sequences and exceptional scattered sequences 
can be seen as the geometrical counterparts of %exceptional 
MRD codes.
Rank-distance (RD) codes were introduced already in the late 70's by Delsarte \cite{delsarte1978bilinear} and then rediscovered by Gabidulin a few years later \cite{gabidulin1985theory}. Due to their applications in network coding \cite{MR2450762} and cryptography \cite{gabidulin1991ideals,MR3678916}, they attracted lots of attention in the last decade.
RD codes are sets of matrices over a finite field $\mathbb{F}_q$ endowed with the so-called rank distance: the distance  between two elements is defined as the rank of their difference. Among them, of particular interest is the family of rank-metric codes whose parameters are optimal, i.e.  they have the maximum possible cardinality for a given minimum rank. Such codes are called \emph{maximum rank distance (MRD) codes} and constructing new families is an important and active research task.

From a different perspective, rank-metric codes can also  be seen as sets of (restrictions of) $\mathbb{F}_q$-linear homomorphisms from $(\mathbb{F}_{q^n})^m$ to $\mathbb{F}_{q^n}$ equipped with the rank distance; see \cite[Sections 2.2 and 2.3]{bartoli2022exceptional}.

Notably, scattered subspaces can also be described via linearized polynomials. In the  univariate case, such a connection was already exploited in \cite{sheekey2016new} by Sheekey, where the notion of scattered polynomials was introduced; see also \cite{bartoli2018exceptional}. Let $\mathcal{L}_{n,q}[X]$ be the set of $q$-linearized polynomials. For a polynomial $f \in \mathcal{L}_{n,q}[X]$ and a nonnegative integer $t\leq n-1$ we say that $f$ is {scattered of index $t$} if for every $x,y \in \mathbb{F}_{q^n}^*$
\[ \frac{f(x)}{x^{q^t}}=\frac{f(y)}{y^{q^t}}\,\, \Longleftrightarrow\,\, \frac{y}x\in \mathbb{F}_q, \]
or equivalently
\begin{equation*} \dim_{\mathbb{F}_q}(\ker(f(x)-\alpha x^{q^t}))\leq 1, \,\,\,\text{for every}\,\,\, \alpha \in \mathbb{F}_{q^n}. \end{equation*}
In a more geometrical setting, a scattered polynomial is connected with a scattered subspace of the projective line; see \cite{blokhuis2000scattered}. From a coding theory point of view,  $f$ is scattered of index $t$ if and only if $\mathcal{C}_{f,t}=\langle x^{q^t}, f(x) \rangle_{\mathbb{F}_{q^n}}$ is an MRD code with $\dim_{\mathbb{F}_{q^n}}(\mathcal{C}_{f,t})=2$.
The polynomial $f$ is said to be \emph{exceptional scattered} of index $t$ if it is scattered of index $t$ as a polynomial in $\mathcal{L}_{\ell n,q}[X]$, for infinitely many $\ell$; see \cite{bartoli2018exceptional}. The classification of exceptional scattered polynomials is still not complete, although it gained the attention of several researchers \cite{Bartoli:2020aa4,bartoli2018exceptional,MR4190573,MR4163074,bartoli2022towards}.

So far, many families of scattered polynomials have been constructed; see \cite{smaldore2024scattered, bartoli2023scattered, sheekey2016new,lunardon2018generalized,lunardon2000blocking,zanella2019condition,MR4173668,longobardi2021linear,longobardi2021large,NPZ,zanella2020vertex,csajbok2018new,csajbok2018new2,marino2020mrd,blokhuis2000scattered}. 

Among them, only two families are exceptional:
\begin{itemize}
    \item[(Ps)] $f(x)=x^{q^s}$ of index $0$, with $\gcd(s,n)=1$ (polynomials of so-called pseudoregulus type);
    \item[(LP)] $f(x)=x+\delta x^{q^{2s}}$ of index $s$, with $\gcd(s,n)=1$ and $\mathrm{N}_{q^n/q}(\delta)\ne1$ (so-called LP polynomials).
\end{itemize}

The generalization of the notion of exceptional scattered polynomials, together with  their connection with $\mathbb{F}_{q^n}$-linear MRD codes of $\mathbb{F}_{q^n}$-dimension $2$, yielded the introduction of the concept of 
 $\mathbb{F}_{q^n}$-linear MRD codes of \emph{exceptional type}; see \cite{bartoli2021linear}. An $\mathbb{F}_{q^n}$-linear MRD code $\mathcal{C}\subseteq\mathcal{L}_{n,q}[X]$ is an \emph{exceptional MRD code} if the rank metric code
\[ \mathcal{C}_\ell=\langle\mathcal{C}\rangle_{\mathbb{F}_{q^{\ell n}}}\subseteq \mathcal{L}_{\ell n,q}[X] \]
is an MRD code (whose left idealizer is equivalent to $\{\alpha x\colon \alpha \in\mathbb{F}_{q^n}\}$) for infinitely many $\ell$.

Only two families of exceptional $\mathbb{F}_{q^n}$-linear MRD codes are known:
\begin{itemize}
    \item[(G)] $\mathcal{G}_{k,s}=\langle x,x^{q^s},\ldots,x^{q^{s(k-1)}}\rangle_{\mathbb{F}_{q^n}}$, with $\gcd(s,n)=1$; see \cite{delsarte1978bilinear,gabidulin1985theory,kshevetskiy2005new};
    \item[(T)] $\mathcal{H}_{k,s}(\delta)=\langle x^{q^s},\ldots,x^{q^{s(k-1)}},x+\delta x^{q^{sk}}\rangle_{\mathbb{F}_{q^n}}$, with $\gcd(s,n)=1$ and $\mathrm{N}_{q^n/q}(\delta)\neq (-1)^{nk}$; see  \cite{sheekey2016new,lunardon2018generalized}. 
\end{itemize}
The first family is known as \emph{generalized Gabidulin codes} and the second one as \emph{generalized twisted Gabidulin codes}, whereas in \cite{bartoli2018exceptional} it has been shown that the only exceptional $\mathbb{F}_{q^n}$-linear MRD codes spanned by monomials are the codes (G), in connection with so-called \emph{Moore exponent sets}. Non-existence results on exceptional MRD codes were provided in \cite[Main Theorem]{bartoli2021linear}.

A generalization of MRD codes of exceptional type is connected with the  notions of $h$-scattered sequences and exceptional $h$-scattered sequences, introduced in \cite{bartoli2022exceptional} as sequences of multivariate linearized polynomials $f_1,\ldots, f_s\in \mathcal{L}_{n,q}[X_1,\ldots,X_m]$, such that there exists 
$\mathcal I=(i_1,\ldots, i_m) \in \mathbb{N}^m$ so that the space 
$$\{ (x_1^{q^{i_1}}, \ldots,x_m^{q^{i_m}},f_1(x_1,\ldots,x_m), \ldots,f_s(x_1,\ldots,x_m) ) \, : \, x_1,\ldots,x_m \in \mathbb{F}_{q^n}\}$$
is $h$-scattered. 

Using this new terminology, exceptional $\mathbb{F}_{q^n}$-linear MRD codes correspond to exceptional scattered sequences of order 1. In  \cite{bartoli2022exceptional} exceptional scattered sequences of order 2 were investigated for the first time. Clearly, when considering sequences of order larger than one, one must check that these examples are really "new", i.e. they cannot be obtained as direct sum of two scattered sequences of smaller order. This led to the notion of indecomposability; see \cite{bartoli2022exceptional}.

Recently, in \cite{ABNR22} a new class of rank-metric codes has been introduced and investigated{, the family} of \emph{minimal rank-metric codes}. The peculiarity of these codes is that the set of supports -- which in the rank-metric setting is given by $\mathbb{F}_{q}$-subspaces of $\mathbb{F}^n_{q}$ -- form an antichain with respect to the {inclusion } {and they} are the $q$-analogues of \emph{minimal linear codes} in the Hamming metric, which have been shown to have interesting  combinatorial and geometric properties  and applications to secret sharing schemes, as proposed by Massey {\cite{Massey1993,Massey1995}}. Using the geometric viewpoint, minimal rank-metric codes have been shown in \cite{ABNR22} to correspond to linear cutting blocking sets, which are special $q$-systems $U$ with the following intersection property: for every $\mathbb{F}_{q^m}$-hyperplane {$H$}, it holds $\langle U\cap H\rangle_{\mathbb{F}_{q^m}}=H$.
Furthermore, always in \cite{ABNR22}, constructions of minimal rank-metric codes were proposed using \emph{scattered subspaces}  {in a $3$-dimensional $\mathbb{F}_{q^m}$-space.}

In this paper we consider sequences of order $m\geq 3$ and we exhibit the first example of indecomposable exceptional scattered sequence of order $m$. The equivalence issue is also considered and it is worth mentioning that our family is quite large since it contains many non-equivalent sequences. The MRD codes associated to these sequences are also minimal.

\section{Main Results}
Let $q=p^h$, where $p$ is a prime and $h>0$ an integer, and denote by $\mathbb{F}_q$ the finite field with $q$ elements. 

We start with the definition of scattered sequences.

\begin{definition}\cite[Definition 3.1]{bartoli2022exceptional}\label{Def:ScatteredSequence}
 Consider $\mathcal{F}=(f_1,\ldots,f_s)$, with $f_1,\ldots,f_s\in \mathcal{L}_{n,q}[\underline{X}]$, which is the set of $q$-linearized polynomials in $\underline{X}=(X_1,\dots,X_m)$.
 
 We define $$U_{\mathcal{F}}:=\{(f_1(x_1,\ldots,x_m), \ldots,f_s(x_1,\ldots,x_m) ) \, : \, x_1,\ldots,x_m \in \mathbb{F}_{q^n}\}.$$
 
 Let $\mathcal{I}:=(i_1,i_2,\dots,i_m)\in (\mathbb{Z}/n\mathbb{Z})^m$, we define the \textbf{$\mathcal I$-space}
$U_{\mathcal I,\mathcal{F}}:=U_{\mathcal{F}'},$
where
$$\mathcal F'=(X_1^{q^{i_1}},\ldots, X_m^{q^{i_m}}, f_1,\ldots, f_s).$$
The $s$-tuple $\mathcal{F}:=(f_1,\ldots,f_s)$ is said to be an \textbf{$(\mathcal I;h)_{q^n}$-scattered sequence of order $m$} if the $\mathcal I$-space
$U_{\mathcal I,\mathcal{F}}$
is maximum $h$-scattered in $V(m+s,q^{n})$. An $(\mathcal I;h)_{q^n}$-scattered sequence  $\mathcal{F}:=(f_1,\ldots,f_s)$ of order $m$  is said to be \textbf{exceptional} if it is $h$-scattered over infinitely many extensions $\mathbb{F}_{q^{n\ell}}$ of $\mathbb{F}_{q^n}$.

\end{definition}

The main issue, when considering scattered sequences of order larger than one, is given by its indecomposability.
\begin{definition}
 An $nm$-dimensional $\mathbb{F}_q$-subspace $U_{\mathcal{H}}$ of $V(k,q^n)$   is said to be \textbf{decomposable} if it can be written as
 $$U_{\mathcal{H}}=U_{\mathcal F}\oplus U_{\mathcal G}$$
 for some nonempty $\mathcal F, \mathcal G$. 
 When this happens we say that $\mathcal{F}$ and $\mathcal{G}$ are \textbf{factors} of $\mathcal{H}$. Furthermore, 
 $U$ is then said to be \textbf{indecomposable} if it is not decomposable.
\end{definition}

Let $\mathcal I:=(i_1,\ldots,i_m)$, $\mathcal J:=(j_1,\ldots,j_{m^{\prime}})$, let $\mathcal{F}=(f_1,\ldots,f_{s})$ and $\mathcal{G}=(g_1,\ldots,g_{s^{\prime}})$ be $(\mathcal I;h)_{{q^n}}$ and $(\mathcal J;h)_{{q^n}}$-scattered sequences of orders $m$ and $m^{\prime}$, respectively.
The {direct sum} $\mathcal{H}:=\mathcal{F}\oplus \mathcal{G}$ is the $(s+s^{\prime})$-tuple $(f_1,\ldots,f_s,g_1,\ldots,g_{s^{\prime}})$. Since 
$$U_{\mathcal I\oplus\mathcal J,\mathcal{H}}=U_{\mathcal I,\mathcal{F}}\oplus U_{\mathcal J,\mathcal{G}},$$
$\mathcal{H}$ is an $(\mathcal I \oplus \mathcal J;h)_{q^{n}}$-scattered sequence of order $m+m^{\prime}$.

The main achievement of this paper is the first example of indecomposable exceptional scattered sequences of order larger than two. 
\begin{definition}Let $n,m$ be positive integers, with $m\geq3$
and $q$ a prime power. Consider the finite field ${\mathbb {F}}_{{q^n}}$. For each choice of  $\alpha_1 ,\dots, \alpha_m\in{\mathbb {F}}_{{q^n}}^*$ and  $I,J\in\mathbb{N}$, $I<J<n$ we define the set:
$$U^{I,J}_{\textbf{A}}:=\{(x_1,\dots,x_m,f_1(\underline{x}),f_2(\underline{x}),\dots,f_{m-1}(\underline{x}),f_m(\underline{x})) : x_1,\dots,x_m\in\field\},$$
where $\textbf{A}:=(\alpha_1,\dots,\alpha_m)$, $f_m(\underline{x}):=x_m\qi+\alpha_1x_1\qj$ and $f_i(\underline{x}):=x_i\qi+\alpha_{i+1}x_{i+1}\qj$ with $i=1,\dots,m-1$. 
\end{definition}

From now on, we will denote $J-I$ as $K$ and $(q^{h\ell}-1)/(q^h-1)$ as $C_{h,\ell}$.
\begin{theorem}\label{thm1}
Assume that $\gcd(I,J)=1$ and that $$K_{\textbf{A}}^{I,J}:=\frac{\alpha_3\cdot\alpha_4^{q^K}\cdot\alpha_5^{q^{2K}}\dots\alpha_m^{q^{(m-3)K}}\cdot\alpha_1^{q^{(m-2)K}}}{\alpha_2^{C_{K,m-1}}}$$ is not a $C_{K,m}$-power in $\mathbb{F}_{q^n}$. Then the set $U^{I,J}_{\textbf{A}}$ is scattered.
\end{theorem}
\begin{proof} Let $\lambda\in\mathbb{F}_{q^n}\setminus\mathbb{F}_q$ be such that
$$(u_1,\dots,u_m,f_1(\underline{u}),\dots,f_m(\underline{u}))=\lambda(x_1,\dots,x_m,f_1(\underline{x}),\dots,f_m(\underline{x})),$$
where $\underline{x}=(x_1,\dots,x_m),\underline{u}=(u_1,\dots,u_m)\in\field^m$. $\us$ is scattered if and only if the equation holds only for $(x_1,\dots,x_m)=\underline 0$.

By way of contradiction, we assume $\underline{x}\neq\underline 0$. We have the system
$$\begin{cases}
\lambda x_1\qi+\lambda\alpha_2x_2\qj=\lambda\qi x_1\qi+\lambda\qj\alpha_2x_2\qj\\
    \vdots\\
    \lambda x_i\qi+\lambda\alpha_{i+1}x_{i+1}\qj=\lambda\qi x_i\qi+\lambda\qj\alpha_{i+1}x_{i+1}\qj, \hspace{8 mm} i=2,\dots,m-1\\
    \vdots\\
    \lambda x_m\qi+\lambda\alpha_1x_1\qj=\lambda\qi x_m\qi+\lambda\qj\alpha_1x_1\qj.\\
\end{cases}$$
So we obtain
$$\begin{cases}
    (\lambda\qi -\lambda)x_1\qi+(\lambda\qj-\lambda)\alpha_2 x_2\qj=0\\
    \vdots\\
    (\lambda\qi -\lambda)x_i\qi+(\lambda\qj-\lambda) \alpha_{i+1}x_{i+1}\qj=0,\hspace{8 mm} i=2,\dots,m-1\\
    \vdots\\
    (\lambda\qi -\lambda)x_m\qi+(\lambda\qj-\lambda)\alpha_1 x_1\qj=0.
\end{cases}$$
It is easy to see that a nontrivial solution $\underline x$ must satisfy  $x_1\cdot x_2\cdots x_{m-1}\cdot x_m\neq 0$.

We note that this is a linear system in the variables $(\lambda\qi -\lambda) and (\lambda\qj-\lambda)$, and since $\lambda\not\in\mathbb{F}_q$, $(\lambda\qi -\lambda)\neq 0$ or $(\lambda\qj -\lambda)\neq 0$ holds. The above system has a nontrivial solution and thus all the $2\times 2$ submatrices of
$$M:=\begin{pmatrix}
    x_1\qi&\alpha_2 x_2\qj\\
    x_2\qi&\alpha_3 x_3\qj\\
    \vdots&\vdots\\
    x_i\qi&\alpha_{i+1} x_{i+1}\qj\\
    \vdots&\vdots\\
    x_{m-1}\qi&\alpha_m x_m\qj\\
    x_m\qi&\alpha_1 x_1\qj
\end{pmatrix}$$
are singular, in particular 
\begin{align*}
    x_1\qi\alpha_ix_i\qj-x_{i-1}\qi\alpha_2x_2\qj=0,&\hspace{8 mm} i=3,\dots,m,\\
    x_1^{q^I+q^J}\alpha_1-\alpha_2x_m\qi x_2\qj=0.\phantom{...}&
\end{align*}
Letting $y_i:=x_i\qi/x_1\qi $ for $ i=2,\dots,m$ and dividing the equations by $x_1^{q^I+q^J}$ we get
\begin{align*}
   (E_i)\hspace{20mm} \alpha_iy_i\qk-y_{i-1}\alpha_2y_2\qk=0,\hspace{8 mm} i=3,\dots,m,&\phantom{ffffffffffffff}\\
    (E)\hspace{23mm}\alpha_1-\alpha_2y_m y_2\qk=0.\phantom{......rrrrrrrrrrrr.rrr}
\end{align*}
From $(E_m)$ we obtain $$y_{m-1}=\frac{\alpha_m}{\alpha_2}\frac{y_m\qk}{y_2\qk}.$$
Substituting \(y_{m - 1}\) in \((E_{m-1})\), we obtain 
$$y_{m-2}=\frac{\alpha_{m-1}\alpha_m\qk}{\alpha_2^{q^K+1}}\frac{y_m^{q^{2K}}}{y_2^{q^{2K}+q^K}}.$$
We can iterate this procedure and get
\begin{equation}\label{eqlunga}
    y_{m-(m-2)}=y_2=\frac{\alpha_3\cdot\alpha_4\qk\dots\alpha_m^{q^{(m-3)K}}}{\alpha_2^{C_{K,m-2}}}\frac{y_m^{q^{(m-2)K}}}{y_2^{C_{K,m-1}-1}}.
    \end{equation}
From $(E)$ we obtain $y_m=\alpha_1/(\alpha_2y_2\qk)$ and substituting in (\ref{eqlunga}) we get
$$y_2=\frac{\alpha_3\cdot\alpha_4\qk\dots\alpha_m^{q^{(m-3)K}}}{\alpha_2^{C_{K,m-2}}}\frac{\alpha_1^{q^{(m-2)K}}}{\alpha_2^{q^{(m-2)K}}}\frac{1}{y_2^{C_{K,m-1}-1}\cdot y_2^{q^{(m-1)K}}}.$$
Thus $$y_2=\frac{\alpha_3\cdot\alpha_4\qk\dots\alpha_m^{q^{(m-3)K}}\cdot \alpha_1^{q^{(m-2)K}}}{\alpha_2^{C_{K,m-1}}}\frac{1}{y_2^{C_{K,m}-1}}.$$
Thus $$y_2^{C_{K,m}}=K_{\textbf{A}}^{I,J},$$ a contradiction to our hypothesis. Therefore 
$U^{I,J}_{\textbf{A}}$ is scattered. 
\end{proof}
We now prove that $U^{I,J}_{\textbf{A}}$ is exceptional scattered. 
\begin{prop}\label{propTN} 
    Let $n\in\mathbb{N}$, let $A \in \mathbb{N}$ such that $\gcd(q,A)=1$, then there exist infinitely many $h \in \mathbb{N}$ such that $\gcd(A,C_{n,h})=1$.
\end{prop}
    \begin{proof}
    Consider the factorization of $A=p_1\cdot p_2\cdots p_N$. By the way of contradiction suppose that there are not infinitely many $h \in \mathbb{N}$ such that $\gcd(A,C_{n,h})=1$.
    Therefore, there must exist an $\overline{h}$ such that $$\forall j>0 \text{ there exists an } i_j\text{ such that }p_{i_j}|f(j):=\frac{q^{n(\overline{h}+j)}-1}{q^n-1}=1+q^n+\cdots+q^{n(\overline{h}-1+j)}.$$
We select the primes from the factorization of $A$ that divide at least one $f(j)$ and we denote them by $\{p_1, \ldots, p_M\}$.

We note that \begin{itemize}
    \item if $p_i|f(j)$ then $p_i\not |f(j+1)$ since $p_i\nmid q$;
    \item if $p_i|f(j)$ then $p_i \mid f(kj+(k-1)\overline{h})$ \text{ for all } $k>0$.
\end{itemize}
 We prove second property by induction:
 \begin{itemize}
     \item[-] $k=2$: Since $p_i \mid q^{n(\overline{h}+j)}(1+q^n+\dots+q^{n(\overline{h}-1+j)})$, we have that $p_i\mid f(2j+\overline{h})=1+q^n+\dots+q^{n({\overline{h}}-1+j)}+q^{n(\overline{h}+j)}+\cdots+q^{n(2\overline{h}-1+2j)}$.
\item[-]$P(k-1)\to P(k)$:
The proof is analogous to the base case noticing that $f(kj+(k-1)\overline{h})=f((k-1)j+(k-2)\overline{h})+q^{n(k-1)(\overline{h}+j)}f(j).$
 \end{itemize}

 Let $j_i>0$ be such that $p_i|f(j_i)$ for $i=1,\dots,M$. This implies that $p_i \mid f(kj_i+(k-1)\overline{h})$ for all $k>0$.
 Note that  $kj_i+(k-1)\overline{h}=k(j_i+\overline{h})-\overline{h}$ and thus   $p_i \mid f(\prod_{i=1}^M(j_i+\overline{h})-\overline{h})$, for each $ i=1,\dots,M$. So $p_i\nmid f(\prod_{i=1}^M(j_i+\overline{h})-\overline{h}+1)$ for each $i=1,\dots,M$, a contradiction.
 \end{proof}

\begin{cor} \label{cor1}
Assume that $\gcd(I,J)=1$ and that
$K_{\textbf{A}}^{I,J}$
is not an $C_{K,m}$-power in $\mathbb{F}_{q^n}$. Then the set $U^{I,J}_{\textbf{A}}$ is exceptional scattered.
\end{cor}
\begin{proof}
    From the previous proposition, there exist infinitely many $h\in\mathbb{N}$ such that 
    $$\gcd(C_{K,m},C_{n,h})=1.$$ 
    Let us consider a fixed $h$ satisfying the above property. By Bézout's identity, there exist integers $c_1$ and $c_2$ such that $c_1C_{K,m}+c_2C_{n,h}=1$. Suppose by the way of contradiction that there exists  $\xi \in \mathbb{F}_{q^{hn}}\setminus\mathbb{F}_{q^n}$ such that $K_{\textbf{A}}^{I,J}=\xi^{C_{K,m}}$. So $\xi^{C_{K,m}}\in\mathbb{F}_{q^n}$, and so $1=(\xi^{C_{K,m}})^{q^n-1}=(\xi^{q^n-1})^{C_{K,m}}$.

Raising both sides to the power $c_1$, we obtain $$1=(\xi^{q^n-1})^{c_1C_{K,m}}=(\xi^{q^n-1})^{-c_2C_{n,h}}(\xi^{q^n-1})=\xi^{q^n-1},$$
a contradiction to $\xi \notin \mathbb{F}_{q^n}$. 

Therefore, there are infinitely many extensions of $\mathbb{F}_q$ such that $K_{\textbf{A}}^{I,J}$ is not a $C_{K,m}$-power, and by  Theorem \ref{thm1} the claim follows.
\end{proof}

\vspace{2mm}

The evasiveness property will be crucial to prove the indecomposability of $\us$.
\begin{definition}
Let $h,r,k,n$ be positive integers, such that $h<k$ and $h \le r$. An $\mathbb{F}_q$-subspace $U\subseteq V(k,q^n)$ is said to be {$(h,r)$-evasive} if for every $h$-dimensional $\field$-subspace  $H\subseteq V(k,q^n)$, it holds $\dim_{\mathbb{F}_q}(U\cap H)\leq r$. When $h=r$, an  $(h,h)$-evasive subspace is called {$h$-scattered}. Furthermore, when $h=1$,  a $1$-scattered subspace is simply called {scattered}.
\end{definition}
\begin{theorem} \label{thm2}
    If $n\geq 2(mJ+J+1)$ then $U^{I,J}_{\mathbf{A}}$ is $(r,\frac{rn}{2}-1)_q$-evasive for any odd $r\in[2,\dots,m]$.
\end{theorem}
\begin{proof}
Let $r\in[2,\dots,m]$, and let $H$ be a subspace of $\field^{2m}$ of dimension $r$. It can be defined by $2m-r$ equation in $z_1,\dots,z_{2m}$, let $t$ be the number of equation containing only the first $m$ variables, and let $s$ be the number of the remaining equations. Since $r$ is odd and $r=2m-s-t$, we can observe that $t\neq s$. The subspace $H$ can be described as the solution oh a homogeneous linear system whose matrix is of type
$$
    \begin{pmatrix}
    A&E\\
    D&0
\end{pmatrix},
$$
with $A,E\in M_{s,m},D\in M_{t,m}$ and $E,D$ are full rank matrix. Consider the substitution 
$$(z_1,\dots,z_{2m})\longrightarrow (x_1,\dots,x_m,f_1(x_1,\dots,x_m),\dots,f_m(x_1,\dots,x_m)).$$
We need to count how many $(x_1,\dots,x_m)\in\mathbb{F}^m_{q^n}$ satisfy the system. After the substitution we see the system in the unknowns $x_1,\dots,x_m,x_1\qi,\dots,x_m\qi,x_1\qj,\dots,x_m\qj.$ The new matrix of this system is
$$
    \begin{pmatrix}
    A&B&C\\
    D&0&0
\end{pmatrix},
$$
where $B,C\in M_{s,m}$ are the matrices associated with the unknowns $(x_1\qi,\dots,x_m\qi)$ and $(x_1\qj,\dots,x_m\qj)$. Note that $B=E$ and $C$ is obtained by shifting the columns of $E$ to the right by one position and then multiplying each $\ell$-th column by $\alpha_{\ell}$ for each $\ell\in [1,...,m]$.

We consider two cases.

\textbf{Case 1.} ${t>s}$

Considering the $t$ linear equations in $x_1\dots x_m$, they, since the matrix $D$ is full-rank, we can bound the total number of solutions by $q^{n(m-t)}$. 

We have \begin{eqnarray*}
    q^{n(m-t)}= q^{\frac{n}{2}(2m-2t)}\leq q^{\frac{n}{2}(2m-s-t)-1}=q^{\frac{rn}{2}-1}.
\end{eqnarray*}

\textbf{Case 2.} ${t<s}$

We focus on the $s$ semilinear equations in $x_1\dots x_m$. We consider the unknowns as $x_{ij}=x_i^{q^j}$. We can apply $n-J-1$ Frobenius, that is, raising to the power of $q$, to the $s$ equations and obtain a system of $s(n-J)$ equations with unknowns $x_{ij}$ and matrix 
$$\overline{M}:=
    \begin{pmatrix}
    A&B&C&0&0&\cdots&0\\
    0&A^q&B^q&C^q&0&\cdots&0\\
    \vdots&\vdots&\vdots&\vdots&\vdots&\vdots&\vdots&\\
    0&0&\cdots&0&A^{q^{n-J-1}}&B^{q^{n-J-1}}&C^{q^{n-J-1}}
\end{pmatrix},
$$
where $A^{q}$ means the matrix having as entries the $q$
 powers of the entries in $A$.

 Since $C$ is a full-rank matrix, $\overline{M}$ is of full rank as well. The total number of solutions $(\overline{x}_{11},\dots,\overline{x}_{mn})$ is $q^{mn-(n-J)s}$. 
 
 Note that based on the assumption on $n$, we have 
 \begin{equation*}
     m\leq\frac{\frac{n}{2}-J-1}{J},
 \end{equation*}
 hence, since $2m\geq s+t$ and $s>t$ we have $m>t$, so 
 \begin{equation}\label{eqt}
     t\leq\frac{\frac{n}{2}-J-1}{J}.
 \end{equation}
 We obtain
 \begin{eqnarray*}
     q^{mn-(n-J)s}&=&q^{\frac{n}{2} (2m-2s)+Js}=q^{\frac{n}{2} (2m-s-t)-1+{n \over 2}(t-s)+Js+1}\\
     &=&q^{\frac{n}{2} (2m-s-t)-1}q^{s(J-\frac{n}{2})+{n\over2}t+1}\\
     &\leq&q^{\frac{n}{2} (2m-s-t)-1}q^{(t+1)(J-\frac{n}{2})+{n\over2}t+1}\\
     &\leq&q^{\frac{n}{2} (2m-s-t)-1},
 \end{eqnarray*}
 since $(t+1)(J-\frac{n}{2})+{n\over2}t+1\leq 0$ from (\ref{eqt}).
\end{proof}

\begin{theorem}
   If $n\geq 2(mJ+J+1)$ and $\frac{\Pi_{\delta+2}}{\Pi_{2}} \textrm{ is not a $(q^{mK}-1)$-power in $\field$}$  for any $\delta=1,\dots,m-1$, where $\Pi_{i}=\alpha_{i}^{q^{(m-1)K}}\alpha_{i-1}^{q^{(m-2)K}} \cdots \alpha_{i+2}^{q^K}  \alpha_{i+1}$, where the indices of $\alpha_i$ are modulo m in the range $[1,\dots,m]$, then $U^{I,J}_{\mathbf{A}}$ is $(r,\frac{rn}{2}-1)_q$-evasive for any even $r\in[2,\dots,m]$. 
\end{theorem}
\begin{proof}
The first part of the proof is the same as in the previous one; for even $r$, we also need to consider the case where $s=t$.
Considering the $s$ linear equations defined by $D$, we deduce the bound $q^{n(m-s)}$ for the number of solutions $(\overline{x}_1,\dots,\overline{x}_m)$. This means that it is sufficient to prove that the space spanned by the rows of the matrix $D|0|0$ cannot be the same as the one defined by the matrix $A|B|C$. 
Notice that if $A$ and $D$ were not proportional, we would have concluded, so we can consider $A=0$.

Let
$$B=\begin{pmatrix}
H_1&H_2&\ldots&H_m
\end{pmatrix} \qquad 
C= \begin{pmatrix}
\alpha_1 H_m&\alpha_2 H_1&\ldots&\alpha_mH_{m-1}
\end{pmatrix}
$$
constructed as described in Theorem \ref{thm2}. Consider 
$$B (x_1^{q^I},\ldots,x_m^{q^I})^T+C (x_1^{q^J},\ldots,x_m^{q^J})^T=(0,\ldots,0)^T=D(x_1,\ldots,x_m)^T,$$.

If the above system has as solutions $D(x_1,\ldots,x_m)^T=(0,\ldots,0)^T$ then the rowspan of $B^{q^K}$ equals the rowspan of $C$.

This yields the existence of $L \in \GL(s,q^n)$ such that 
$$LB^{q^K}=C,$$
that is 
$$L H_1^{q^K}=\alpha_1 H_m, \; L H_2^{q^K}=\alpha_2 H_1, \ldots, \; L H_{m-1}^{q^K}=\alpha_{m-1} H_{m-2},\; L H_m^{q^K}=\alpha_m H_{m-1},$$
i.e.
$$H_1^{q^K}=\alpha_1 L^{-1}H_m, \; H_2^{q^K}=\alpha_2 L^{-1} H_1, \ldots, \; H_{m-1}^{q^K}=\alpha_{m-1} L^{-1} H_{m-2},\; H_m^{q^K}=\alpha_mL^{-1} H_{m-1}.$$     

Consider $i=\{1..,m\}$.  
\begin{eqnarray*}
H_i^{q^{mK}}&=&\alpha_{i}^{q^{(m-1)K}} (L^{-1})^{q^{(m-1)K}}H_{i-1}^{q^{(m-1)K}}\\&=&\alpha_{i}^{q^{(m-1)K}} (L^{-1})^{q^{(m-1)K}}\alpha_{i-1}^{q^{(m-2)K}} (L^{-1})^{q^{(m-2)K}}H_{i-2}^{q^{(m-2)K}}\\
&=&\alpha_{i}^{q^{(m-1)K}}\alpha_{i-1}^{q^{(m-2)K}} \cdots \alpha_{i+2}^{q^K} (L^{-1})^{q^{(m-1)K}} (L^{-1})^{q^{(m-2)K}} \cdots (L^{-1})^{q^{K}} H_{i+1}^{q^{K}}\\
&=&\alpha_{i}^{q^{(m-1)K}}\alpha_{i-1}^{q^{(m-2)K}} \cdots \alpha_{i+2}^{q^K}  \alpha_{i+1} (L^{-1})^{q^{(m-1)K}} (L^{-1})^{q^{(m-2)K}} \cdots (L^{-1})^{q^{K}}(L^{-1}) H_{i}.
\end{eqnarray*}
Write the above equation as 
$$H_i^{q^{mK}}=\Pi_{i} \overline{L} H_i,$$
where $\overline{L}=(L^{-1})^{q^{(m-1)K}} (L^{-1})^{q^{(m-2)K}} \cdots (L^{-1})^{q^{K}}(L^{-1})$.

Since the rank of $B$ is $s$ and $s<m$ there exist 
linearly independent columns $H_{i_1},\ldots, H_{i_s}$ and write $H_{i_{s+1}}$  as 
$$H_{i_{s+1}}=\lambda_1 H_{i_1}+\cdots+ \lambda_sH_{i_s},$$
for some $\lambda_i \in \mathbb{F}_{q^n}$ not all vanishing (since $H_{i_{s+1}}$ being the zero column yields $B=0$, a contradiction).

Thus 

\begin{eqnarray*}
(\lambda_1 H_{i_1}+\cdots+ \lambda_sH_{i_s})^{q^{mK}}&=&\Pi_{i_{s+1}} \overline{L} (H_{i_{s+1}})=\Pi_{i_{s+1}} \overline{L} (\lambda_1 H_{i_1}+\cdots+ \lambda_sH_{i_s})\\
&=&\Pi_{i_{s+1}} (\lambda_1 H_{i_1}^{q^{mK}}/\Pi_{i_1}+\cdots+ \lambda_s  H_{i_s}^{q^{mK}}/\Pi_{i_s}).
\end{eqnarray*}
Since $H_{i_1}^{q^{mK}},\ldots, H_{i_s}^{q^{mK}}$ are also independent,
$$\lambda_j^{q^{mK}}=\frac{\Pi_{i_{s+1}}}{\Pi_{i_j}}\lambda_j, \qquad \forall j=1,\ldots,s.$$

This contradicts that $\frac{\Pi_{\delta+2}}{\Pi_{2}}$ is not a $(q^{mK}-1)$-power for any $\delta= {1,\ldots,m}$, since $\Pi_{\delta}^{q^{\gamma K}}=(\xi_{\delta,\gamma})^{q^{mK}-1}\Pi_{\delta+\gamma}$ for some $\xi_{\delta,\gamma}\in\field$, so the claim follows.
\end{proof}

\begin{cor}\label{Cor:Finale}
  If $n\geq 2(mJ+J+1)$, and $\frac{\Pi_{\delta+2}}{\Pi_{2}}$ is not a $(q^{mK}-1)$-power in $\mathbb{F}_{q^n}$ for any $\delta=1,\dots,m-1$, then $\us$ is indecomposable.   
\end{cor}
\begin{proof}
    It follows from \cite[Lemma 3.4]{bartoli2022exceptional}.
\end{proof}
  We conclude this section with our main result. 
\begin{theorem}\label{maint}
     Assume that $\gcd(I,J)=1$, $K_{\textbf{A}}^{I,J}$ is not a $C_{K,m}$-power, and $\frac{\Pi_{\delta+2}}{\Pi_{2}}$ is not a $(q^{mK}-1)$-power in $\mathbb{F}_{q^n}$ for any $\delta=1,\dots,m-1$. Then $\us$ is scattered and indecomposable in infinitely many extensions of $\field$.
\end{theorem}
     \begin{proof} By Proposition \ref{propTN} there exists a sequence of positive integers $(h_k)_k$ such that $\gcd(q^{mK}-1,C_{n,h_k})=1$. This implies $\gcd(C_{K.m},C_{n,h_k})=1$, so by Corollary \ref{cor1}, $\us$ is scattered in 
     $\mathbb{F}_{q^{nh_k}}$ for every $k$. With similar calculations, we obtain that $\frac{\Pi_{\delta+2}}{\Pi_{2}}$ is not a $(q^{mK}-1)$-power in $\mathbb{F}_{q^{nh_k}}$ for every $k$ and $\delta$.
     
     Also, there exists an 
     $h_{k_0}$ such that 
     $$nh_{k_0}\geq 2(mJ+J+1)$$ and, by Corollary \ref{Cor:Finale}, $\us$ is indecomposable in every extension $\mathbb{F}_{q^{nh_k}}$ with $h_k\ge h_{k_0}$.
     \end{proof}
\begin{oss}
   If $\frac{\Pi_{\delta+2}}{\Pi_{2}}$ is not a $(q^{mK}-1)$-power in $\mathbb{F}_{q^n}$ for any $\delta=1,\dots,m-1$ then $m|n$.
\end{oss}  
\begin{proof}
Define $d:=\gcd(m,n)$. We have 
\begin{eqnarray*}
    \frac{\Pi_{3}}{\Pi_{2}}&=&\frac{\alpha_{3}^{q^{(m-1)K}}\alpha_{2}^{q^{(m-2)K}} \cdots \alpha_{5}^{q^K}  \alpha_{4}}{\alpha_{2}^{q^{(m-1)K}}\alpha_{1}^{q^{(m-2)K}} \cdots \alpha_{4}^{q^K}  \alpha_{3}}=
    %\frac{\alpha_{3}^{q^{(m-2)K}(q^K-1)}\alpha_{4}^{q^{(m-3)K}(q^k-1)} \cdots \alpha_{1}^{q^K-1}}{\alpha_2^{q^{(m-1)K}-1}}=
    \frac{\left(\alpha_{3}^{-q^{(m-1)K}}\alpha_{2}^{-q^{(m-2)K}} \cdots \alpha_{4}^{-1}\right)^{q^K-1}}{\alpha_3^{1-q^{mK}}}\\
    \frac{\Pi_{4}}{\Pi_{2}}&=&\frac{\alpha_{4}^{q^{(m-1)K}}\alpha_{3}^{q^{(m-2)K}} \cdots \alpha_{6}^{q^K}  \alpha_{5}}{\alpha_{2}^{q^{(m-1)K}}\alpha_{1}^{q^{(m-2)K}} \cdots \alpha_{4}^{q^K}  \alpha_{3}}=
    %\frac{\alpha_{4}^{q^{(m-3)K}(q^{2K}-1)}\alpha_{5}^{q^{K(m-4)}(q^{2K}-1)} \cdots \alpha_{1}^{q^{2K}-1}}{\alpha_2^{q^K(q^{(m-2)K}-1)}\alpha_3^{q^{(m-2)K}-1}}=
    \frac{\left(\alpha_{4}^{-q^{(m-1)K}}\alpha_{3}^{-q^{(m-2)K}} \cdots \alpha_{6}^{-q^K}\alpha_5^{-1}\right)^{q^{2K}-1}}{\left(\alpha_4^{q^K}\alpha_3\right)^{1-q^{mK}}}\\
    &\vdots&\\
    \frac{\Pi_{m}}{\Pi_{2}}&=&\frac{\alpha_{m}^{q^{(m-1)K}}\alpha_{m-1}^{q^{(m-2)K}} \cdots \alpha_{2}^{q^K}  \alpha_{1}}{\alpha_{2}^{q^{(m-1)K}}\alpha_{1}^{q^{(m-2)K}} \cdots \alpha_{4}^{q^K}  \alpha_{3}}=
    %\frac{{\alpha_{m}^{q^K(q^{(m-2)K}-1)}\alpha_{1}^{q^{(m-2)K}-1}}}{\alpha_2^{q^{(m-3)K}(q^{2K}-1)}\alpha_3^{q^{K(m-4)}(q^{2K}-1)}\cdots\alpha_{m-1}^{q^{2K}-1}}
    \frac{\left(\alpha_{m}^{-q^{(m-1)K}}\alpha_{m-1}^{-q^{(m-2)K}} \cdots \alpha_{2}^{-q^K}\alpha_1^{-1}\right)^{q^{(m-2)K}-1}}{\left(\alpha_m^{q^{(m-3)K}}\alpha_{m-1}^{q^{(m-4)K}}\cdots\alpha_4^{q^K}\alpha_3\right)^{1-q^{mK}}}\\
    \frac{\Pi_{1}}{\Pi_{2}}&=&\frac{\alpha_{1}^{q^{(m-1)K}}\alpha_{m}^{q^{(m-2)K}} \cdots \alpha_{3}^{q^K}  \alpha_{2}}{\alpha_{2}^{q^{(m-1)K}}\alpha_{1}^{q^{(m-2)K}} \cdots \alpha_{4}^{q^K}  \alpha_{3}}=
    \frac{\left(\alpha_{1}^{-q^{(m-1)K}}\alpha_{m}^{-q^{(m-2)K}} \cdots \alpha_{3}^{-q^K}\alpha_2^{-1}\right)^{q^{(m-1)K}-1}}{\left(\alpha_{1}^{q^{(m-2)K}}\alpha_m^{q^{(m-3)K}}\cdots\alpha_4^{q^K}\alpha_3\right)^{1-q^{mK}}}.
\end{eqnarray*}
So we obtain \begin{eqnarray*}
\frac{\Pi_\delta}{\Pi_2}=\frac{\xi_{\delta}}{\Delta^{q^{(\delta-2)K}-1}},&& \frac{\Pi_1}{\Pi_2}=\frac{\xi_{1}}{\Delta^{q^{(m-1)K}-1}},
\end{eqnarray*}
for any $\delta=3,\dots,m$, where $\Delta=\left(\alpha_{2}^{q^{(m-1)K}}\alpha_{1}^{q^{(m-2)K}} \cdots \alpha_{4}^{q^K}\alpha_3\right)$, and $\xi_1,\xi_3,\dots,\xi_m$ are $(q^{mK}-1)$-power, so $\Delta$ is not a $(q^{jK}-1)$-power for any $j=1,\dots,m-1$. 

    Observe that \begin{equation*}
    \Delta^{q^{jK}-1}={\Delta^{-q^{jK}(q^{(m-j)K}-1)}}{\Delta^{q^{mK}-1}},\end{equation*}
    thus, we can restrict our investigation to $j=1,\dots,\lceil \frac{m-1}{2}\rceil$.

    Let $d=\gcd(m,n)<m$, then $d\leq \lceil \frac{m-1}{2}\rceil$. Let $n=dn'$.

    Since
     \begin{eqnarray*}
        &&{\gcd(q^{dK}-1,q^n-1)\frac{q^n-1}{\gcd(q^n-1,q^{mK}-1)}}={q^{\gcd(dK,dn')-1}\cdot\frac{q^n-1}{q^{\gcd(dn',mK)}-1}}=\\&&
        {q^{d\gcd(K,n')-1}\cdot\frac{q^n-1}{q^{d\gcd(n',\frac{m}{d}K)}-1}}=
        {q^{d\gcd(K,n')-1}\cdot\frac{q^n-1}{q^{d\gcd(n',K)}-1}}=
        {q^n-1},
    \end{eqnarray*}
    then
    \begin{eqnarray*}
        \Delta^{\gcd(q^{dK}-1,q^n-1)\frac{q^n-1}{\gcd(q^n-1,q^{mK}-1)}}=1,\\
    \end{eqnarray*}
   and thus $\Delta^{q^{dK}-1}$ is a  (${q^{mK}-1}$)-power, a contradiction.
    \end{proof}
\begin{prop}
   If $n=mn'$, and $m\not|\, K$ then there are at least \begin{equation*}
    Q^{I,J}_{n,m}:=(q^n-1)^{m-1}\left((q^n-1)-\frac{q^n-1}{\gcd(q^n-1,C_{K,m})}-\sum_{j=1}^{\lceil \frac{m-1}{2}\rceil}(q^n-1)\frac{q^{\gcd(mn',j)}-1}{q^{m\gcd(n',K)}-1}\right)
    \end{equation*} $m$-uples $(\alpha_1,\dots,\alpha_m)$ such that $U^{I,J}_{(\alpha_1,\dots,\alpha_m)}$ is scattered and indecomposable, and $Q^{I,J}_{n,m}$ is positive.
\end{prop}
\begin{proof}

 For any chosen $\alpha_2,\dots,\alpha_m$, the function $\alpha_1\longrightarrow K_{\textbf{A}}^{I,J}$ is a permutation of $\field$. So $K_{\textbf{A}}^{I,J}$ is a $C_{K,m}$-power for $(q^n-1)/(\gcd(q^n-1,C_{K,m}))$ choices of $\alpha_1$. A necessary condition to find a suitable $\alpha_1$ is that $\gcd(q^n-1,C_{K,m})\geq 2$, and a sufficient one is that $m$ does not divide $K$.

Also the function $\alpha_1\longrightarrow \Delta$ is a permutation, so counting the $\alpha_1$ such that $\Delta^{q^{jK}-1}$ is a $(q^{mK}-1)$-power is equivalent to count the solutions in $\field$ of the equation
\begin{equation*}
    x^{\gcd(q^n-1,q^j-1)\frac{q^n-1}{\gcd(q^n-1,q^{mK}-1)}}=1.
\end{equation*}
Remarking that $n=mn'$, the solutions are at most 
\begin{equation*}
    (q^n-1)\frac{q^{\gcd(mn',j)}-1}{q^{m\gcd(n',K)}-1}.
\end{equation*}
So the number of $\alpha_1$ such that $\us$ is not scattered and indecomposable is at most 
\begin{equation}\label{bound1}
\frac{q^n-1}{\gcd(q^n-1,C_{K,m})}+\sum_{j=1}^{\lceil \frac{m-1}{2}\rceil}(q^n-1)\frac{q^{\gcd(mn',j)}-1}{q^{m\gcd(n',K)}-1}.    
\end{equation}

    So there are at least $Q^{I,J}_{n,m}$  $m$-uples $(\alpha_1,\dots,\alpha_m)$ such that $\us$ is scattered and indecomposable.
    
    Consider
\begin{equation}\label{bound2}
    \frac{q^n-1}{2}+\frac{m}{2}(q^n-1)\frac{q^{\frac{m}{2}}-1}{q^{m}-1}=(q^n-1)c^{q,m},
\end{equation}
where \begin{equation*}
    c^{q,m}:=\frac{1}{2}+\frac{m}{2}\cdot\frac{1}{q^{\frac{m}{2}}+1}< 1.
    \end{equation*}
    Since $(\ref{bound1})\leq(\ref{bound2})$,
    \begin{equation*}
        Q^{I,J}_{n,m}\geq(q^n-1)^{m-1}\lceil(q^n-1)(1-c^{q,m})\rceil>0.
    \end{equation*}       
\end{proof}
    \begin{comment}
        \begin{ex}
        With $q=8,m=4,I=1,J=3,n=8$ we have at least $42051863208192100395984101250$ $(\alpha_1,\alpha_2,\alpha_3,\alpha_4)$ such that $\us$ is scattered and indecomposable.
    \end{ex}
    \end{comment}

    We will use the results obtained in \cite{bartoli2022new} to prove the minimality of the MRD codes associated to the scattered sequences $(f_1,\dots,f_m)$, indeed \cite[Theorem 2.15]{bartoli2022new} says that the code associated to $\us$ is minimal if and only if $\us$ is cutting. First we need a result on the evasiveness of $U^{I,J}_{\mathbf{A}}$.

   \begin{theorem}\label{ThCutting}
   If there exists $\overline{\delta}\in[1,\dots,m-1]$ such that $\frac{\Pi_{\overline{\delta}+2}}{\Pi_{2}}$ is not a $(q^{mK}-1)$-power in $\mathbb{F}_{q^n}$, then $U^{I,J}_{\mathbf{A}}$ is $(2m-2,mn-(2n-2J))_q$-evasive. 
\end{theorem}
\begin{proof}
Let $H$ be a subspace of $\mathbb{F}^{2m}_{q^n}$ of dimension $2m-2$ and consider its equations in $z_1,\dots,z_{2m}$. Let $t$ be the number of equations containing only the first $m$ variables, and let $s$ be the number of the remaining equations. We consider three cases.

\textbf{Case 1.} ${t=2,s=0}$

We can easily bound the size of $H\cap U^{I,J}_{\mathbf{A}}$ by $q^{n(m-2)}$. Obviously, $n(m-2)\leq mn-2n+2J$.

\textbf{Case 2.} ${t=0,s=2}$

We focus on the $2$ semilinear equations in $x_1\dots x_m$. In order to bound the size of $H\cap U^{I,J}_{\mathbf{A}}$, we consider the unknowns as $x_{ij}=x_i^{q^j}$. We can apply $n-J-1$ Frobenius to the $2$ equations defining $H\cap U^{I,J}_{\mathbf{A}}$ and obtain a system of $2(n-J)$ equations with unknowns $x_{ij}$, with $i=1,\dots,m,j=0,\dots,n$. 
The total number of  solutions $(\overline{x}_{11},\dots,\overline{x}_{mn})$ of such a system  is at most $q^{mn-2(n-J)}$. 
\begin{comment}Given the assumption on $n$, we have $mn-2(n-J)\leq mn-n-1$.
\end{comment}

\textbf{Case 3.} ${t=1,s=1}$

The system describing $H\cap U^{I,J}_{\mathbf{A}}$ reads 
\begin{equation*}
    \begin{cases}
        a_1x_1+\cdots+a_mx_m=0\\
        b_1x_1+\cdots+b_mx_m+c_1(x_1\qi+\alpha_2x_2\qj)+\cdots+c_m(x_m\qi+\alpha_1x_1\qj)=0.
    \end{cases}
\end{equation*}
We can write it as 
\begin{equation*}
    \begin{cases}
        (a_1,\dots,a_m)\cdot(x_1,\dots,x_m)^T=0\\
        (b_1,\dots,b_m)\cdot(x_1,\dots,x_m)^T+(c_1,\dots,c_m)\cdot(x_1\qi,\dots,x_m\qi)^T+\\+(\alpha_1c_m,\dots,\alpha_mc_{m-1})\cdot(x_1\qj,\dots,x_m\qj)^T=0.\\
      
    \end{cases}
\end{equation*}

First, observe that $(c_1\qk,c_2\qk,\ldots,c_m\qk)$ cannot be proportional to $(\alpha_1c_m,\alpha_2c_1,\ldots,\alpha_mc_{m-1})$.
In fact, consider $\ell \in \mathbb{F}^*_{q^n}$  such that 
$$\ell(c_1\qk,c_2\qk,\ldots,c_m\qk)=(\alpha_1c_m,\alpha_2c_1,\ldots,\alpha_mc_{m-1})$$
and observe that $c_i$ being zero for any $i=1,\dots,m$ yields to $(c_1,\dots,c_m)=(0,\dots,0)$, a contradiction.

So
\begin{eqnarray*}
    c_i^{q^{mK}}=\ell^{-q^{(m-1)K}}\alpha_{i}^{q^{(m-1)K}}c_{i+1}^{q^{(m-1)K}}=\cdots=\overline{\ell}\alpha_i^{q^{(m-1)K}}\alpha_{i+1}^{q^{(m-2)K}}\cdots\alpha_{i-1}c_i=\overline{\ell}\Pi_ic_i,   
\end{eqnarray*}
where
$\overline{\ell}=(\ell^{-1})^{q^{(m-1)K}} (\ell^{-1})^{q^{(m-2)K}} \cdots (\ell^{-1})^{q^{K}}(\ell^{-1})$.
Then $\Pi_{\overline{\delta}+2} /\Pi_2$ is a $(q^{mK}-1)-$power, a contradiction.

If $(a_1\qj,\dots,a_m\qj)$ is not proportional to $(\alpha_1c_m,\alpha_2c_1,\ldots,\alpha_mc_{m-1})$, we can consider the system
\begin{equation*}
    \begin{cases}
        (a_1\qj,\dots,a_m\qj)\cdot(x_1\qj,\dots,x_m\qj)^T=0\\
        (b_1,\dots,b_m)\cdot(x_1,\dots,x_m)^T+(c_1,\dots,c_m)\cdot(x_1\qi,\dots,x_m\qi)^T+\\+(\alpha_1c_m,\dots,\alpha_mc_{m-1})\cdot(x_1\qj,\dots,x_m\qj)^T=0,
      
    \end{cases}
\end{equation*}
and we can proceed similarly to \textbf{Case 2.} and obtain the same bound on the number of solutions.

If $(a_1\qj,\dots,a_m\qj)$ is proportional to $(\alpha_1c_m,\alpha_2c_1,\ldots,\alpha_mc_{m-1})$, we can consider the system
\begin{equation*}
    \begin{cases}
        (a_1\qj,\dots,a_m\qj)\cdot(x_1\qj,\dots,x_m\qj)^T=0\\
        (b_1\qk,\dots,b_m\qk)\cdot(x_1\qk,\dots,x_m\qk)^T+(c_1\qk,\dots,c_m\qk)\cdot(x_1\qj,\dots,x_m\qj)^T=0,
      
    \end{cases}
\end{equation*}
and from since $(a_1\qj,\dots,a_m\qj)$ cannot be proportional to $(c_1\qk,c_2\qk,\ldots,c_m\qk)$, we obtain the same bound for $H\cap U^{I,J}_{\mathbf{A}}$.

\end{proof}
 \begin{cor}
 If $n\geq2J+1$ and there exists $\overline{\delta}\in[1,\dots,m-1]$ such that $\frac{\Pi_{\overline{\delta}+2}}{\Pi_{2}}$ is not a $(q^{mK}-1)$-power in $\mathbb{F}_{q^n}$, then $U^{I,J}_{\mathbf{A}}$ is cutting and the code associated to $\us$ is minimal.
 \end{cor}
 \begin{proof}
     The assumption on $n$ guarantees that $mn-n-1\geq mn-(2n-2J)$. By Theorem \ref{ThCutting}, $\us$ is $(2m-2,mn-n-1)_q$-evasive, and by \cite[Theorem 3.3]{bartoli2022new} and \cite[Theorem 2.15]{bartoli2022new} the claim follows.\end{proof}
% \begin{oss}
% \cite[Theorem 2.15]{bartoli2022new} guarantees us that the codes associated to $\us$ are minimal.    
% \end{oss}

\begin{remark} \rm{Given an $\mathbb{F}_q$-subspace we can define its \textit{second generalized weight} as
 $$ d_{rk,2}(U)=n-\max\{\dim_{\mathbb{F}_q}(U\cap W) \,:\, W\subseteq \mathbb{F}^k_q, \ \dim_{\mathbb{F}_{q^n}}(W)=k-2\}.$$
For the second generalized weight, the upper bound in \cite{martinez2016similarities} reads in our case  \begin{equation}\label{Eq:bound}d_{rk,2}(\us)\leq 2n-1.\end{equation}
By Theorem \ref{ThCutting} we know that $d_{rk,2}(\us)\geq 2n-2J$, and as $\us$ is scattered and indecomposable in infinitely many extensions $\mathbb{F}_{q^{n\ell}}$ of $\mathbb{F}_{q^n}$, it is notable that as $n$ grows, the difference between $d_{rk,2}(\us)$ and its upper bound becomes negligible relative to the size of $\mathbb{F}_{q^{n\ell}}$.

%since $\us$ is scattered and indecomposable in infinitely many extension $\mathbb{F}_{q^{n\ell}}$ of $\field$, it is worth noticing that as $n$ increases, the difference between $d_{rk,2}(\us)$ and its upper bound becomes negligible with respect to the size of $\mathbb{F}_{q^{n\ell}}$.
}
\end{remark}

\section{Equivalence issue}
\indent This section is devoted to determine a lower bound on the number of $\Gamma \mathrm{L}_q(2m,q^n)$-inequivalent scattered sets contained in our family. Recall that, since $q=p^h$, the size of $Aut(\mathbb{F}_{q^n})$ is $hn$.
\begin{theorem}
    Let $I,J,I_0,J_0$ be nonnegative integers, such that $J+J_0<n$, $I< J,$ and $I_0< J_0$. The two sets $\us$ and $U_{\mathbf{\overline{A}}}^{I_0,J_0}$ are not $\Gamma \mathrm{L}(2m,q^n)$-equivalent if $(I,J)\neq(I_0,J_0)$.
     \end{theorem}
    \begin{proof} W.l.o.g. we can consider $J\geq J_0$. The two sets $\us$ and $U_{\mathbf{\overline{A}}}^{I_0,J_0}$ are $\Gamma \mathrm{L}(2m,q^n)$-equivalent if and only if there exist $\sigma\in Aut(\mathbb{F}_{q^n})$ and a matrix $M=(a_{ij})_{i,j=1,\dots,2m}\in \GL(2m,q^n)$ such that
    $$M\cdot
    \begin{pmatrix}
        x_1^\sigma\\
        \vdots\\
        x_m^\sigma\\
        (x_1\qi+\alpha_2x_2\qj)^\sigma\\
        \vdots\\
        (x_m\qi+\alpha_1x_1\qj)^\sigma
    \end{pmatrix}
=\begin{pmatrix}
    u_1\\
    \vdots\\
    u_m\\
    u_1\qio+\overline{\alpha_2}u_2\qjo\\
    \vdots\\
    u_m\qio+\overline{\alpha_1}u_1\qjo\\
\end{pmatrix}.
$$
Denote $\widetilde{x}_1:=x_1\as,\dots,\widetilde{x}_m:=x_m\as,A_1:=\alpha_1\as,\dots,A_m:=\alpha_m\as$. We obtain\begin{small} 
    
\begin{equation}\label{system}
\begin{cases}
   \langle \underline{a}_{1},\underline{\widetilde{x}}\rangle+a_{1(m+1)}(\widetilde{x}_1\qi+A_2\widetilde{x}_2\qj)+\cdots+a_{1(2m)}(\widetilde{x}_m\qi+A_1\widetilde{x}_1\qj)=u_1\\
   \vdots\\
   \langle \underline{a}_{m},\underline{\widetilde{x}}\rangle+a_{m(m+1)}(\widetilde{x}_1\qi+A_2\widetilde{x}_2\qj)+\cdots+a_{m(2m)}(\widetilde{x}_m\qi+A_1\widetilde{x}_1\qj)=u_m\\
    \langle \underline{a}_{m+1},\underline{\widetilde{x}}\rangle+a_{(m+1)(m+1)}(\widetilde{x}_1\qi+A_2\widetilde{x}_2\qj)+\cdots+a_{(m+1)(2m)}(\widetilde{x}_m\qi+A_1\widetilde{x}_1\qj)=u_1\qio+\overline{\alpha_2}u_2\qjo\\
   \vdots\\
   \langle \underline{a}_{2m},\underline{\widetilde{x}}\rangle+a_{(2m)(m+1)}(\widetilde{x}_1\qi+A_2\widetilde{x}_2\qj)+\cdots+a_{(2m)(2m)}(\widetilde{x}_m\qi+A_1\widetilde{x}_1\qj)=u_m\qio+\overline{\alpha_1}u_1\qjo,
\end{cases}
\end{equation}
\end{small}
where $\langle \underline{a}_{\delta},\underline{\widetilde{x}}\rangle= a_{(\delta)1}\widetilde{x}_1+\cdots+a_{(\delta) m}\widetilde{x}_m$.
Combining the first, the second and the $(m+1)$-th equations we obtain 
\begin{align*}
a_{(m+1)1}\widetilde{x}_1+\cdots+a_{(m+1)m}\widetilde{x}_m+a_{(m+1)(m+1)}(\widetilde{x}_1\qi+A_2\widetilde{x}_2\qj)+\cdots+a_{(m+1)(2m)}(\widetilde{x}_m\qi+A_1\widetilde{x}_1\qj)&+\\-a^{q^{I_0}}_{11}\widetilde{x}_1\qio-\cdots-a^{q^{I_0}}_{1m}\widetilde{x}_m\qio&+\\-a^{q^{I_0}}_{1(m+1)}(\widetilde{x}_1^{q^{I+I_0}}+A_2\qio\widetilde{x}_2^{q^{J+I_0}})-\cdots-a^{q^{I_0}}_{1(2m)}(\widetilde{x}_m^{q^{I+I_0}}+A_1\qio\widetilde{x}_1^{q^{J+I_0}})&+\\-\overline{\alpha_2}(a^{q^{J_0}}_{21}\widetilde{x}_1\qjo+\cdots+a^{q^{J_0}}_{2m}\widetilde{x}_m\qjo&+\\+a^{q^{J_0}}_{2(m+1)}(\widetilde{x}_1^{q^{I+J_0}}+A_2\qjo\widetilde{x}_2^{q^{J+J_0}})+\cdots+a^{q^{J_0}}_{2(2m)}(\widetilde{x}_m^{q^{I+J_0}}+A_1\qjo\widetilde{x}_1^{q^{J+J_0}}))&=0.\tag{a}\label{eq1}
 \end{align*}
 
 If there exists an element in $\GL(2m,{q^n})$ such that (\ref{eq1}) is satisfied for every $\widetilde{x}_1, \dots, \widetilde{x}_m$, then the polynomial above (in the variables $\widetilde{x}_1, \dots, \widetilde{x}_m$) must vanish. From the assumptions on $I, J, I_0, J_0$, we must have that $a_{(m+1)1}= \dots= a_{(m+1)m} = 0$. We consider a number of cases.
 \begin{itemize}
     \item $I\neq I_0,J_0$
     
     Considering the coefficients of $\widetilde{x}_1\qi,\dots,\widetilde{x}_m\qi$ we obtain $$a_{(m+1)(m+1)}=\cdots=a_{(m+1)2m}=0.$$ Hence $M\not\in \GL(2m,{q^n})$.
     \item $I=J_0$

     \begin{itemize}
         \item $J\neq I+I_0,2I$

         Considering the coefficients of $\widetilde{x}_1\qj,\dots,\widetilde{x}_m\qj$ we obtain $$a_{(m+1)(m+1)}=\cdots=a_{(m+1)2m}=0.$$ Hence $M\not\in \GL(2m,{q^n})$.

         \item $J=I+I_0$ and $J+I_0=2I$

         Considering the coefficients of 
         $\widetilde{x}_1\qj,\dots,\widetilde{x}_m\qj,\widetilde{x}_1^{q^{I+J_0}},\dots,\widetilde{x}_m^{q^{I+J_0}},\widetilde{x}_1^{q^{J+J_0}},\dots,\widetilde{x}_m^{q^{J+J_0}}$ we obtain $$a_{(m+1)(m+1)}=\cdots=a_{(m+1)2m}=0.$$ Hence $M\not\in\GL(2m,{q^n})$.

         \item $J=I+I_0$ and $J+I_0\neq 2I$

         Considering the coefficients of 
         $\widetilde{x}_1\qj,\dots,\widetilde{x}_m\qj,\widetilde{x}_1^{q^{J+I_0}},\dots,\widetilde{x}_m^{q^{J+I_0}}$ we obtain $$a_{(m+1)(m+1)}=\cdots=a_{(m+1)2m}=0.$$ Hence $M\not\in \GL(2m,{q^n})$.

         \item $J=2I$ and $J\neq I+I_0$

         Considering the coefficients of 
         $\widetilde{x}_1\qj,\dots,\widetilde{x}_m\qj,\widetilde{x}_1^{q^{J+J_0}},\dots,\widetilde{x}_m^{q^{J+J_0}}$ we obtain $$a_{(m+1)(m+1)}=\cdots=a_{(m+1)2m}=0.$$ Hence $M\not\in \GL(2m,{q^n})$.
         \end{itemize}
    \item $I=I_0$ and $J\neq J_0$
    
    \begin{itemize}
        \item $J\neq I+J_0$ and $J\neq 2I$

        Considering the coefficients of 
         $\widetilde{x}_1\qj,\dots,\widetilde{x}_m\qj$ we obtain $$a_{(m+1)(m+1)}=\cdots=a_{(m+1)2m}=0.$$ Hence $M\notin \GL(2m,{q^n})$.

        \item $J= I+J_0$

        Considering the coefficients of 
         $\widetilde{x}_1\qj,\dots,\widetilde{x}_m\qj,\widetilde{x}_1^{q^{J+J_0}},\dots,\widetilde{x}_m^{q^{J+J_0}}$ we obtain $$a_{(m+1)(m+1)}=\cdots=a_{(m+1)2m}=0.$$ Hence $M\not\in \GL(2m,{q^n})$.

        \item $J=2I$ and so $J_0\neq3I$

        Considering the coefficients of 
         $\widetilde{x}_1\qj,\dots,\widetilde{x}_m\qj,\widetilde{x}_1^{q^{I+J}},\dots,\widetilde{x}_m^{I+J}$ we obtain $$a_{(m+1)(m+1)}=\cdots=a_{(m+1)2m}=0.$$ Hence $M\not\in \GL(2m,{q^n})$.
    \end{itemize}

 \end{itemize}
 \end{proof}

 \begin{theorem}  
 Let $(I,J)$ be such that $J<n/2$, $\gcd(I,J)=1$. Two sets $\us$ and $U_{\mathbf{\overline{A}}}^{I,J}$ are $\Gamma L(2m,q^n)$-equivalent if and only if there exists $\sigma\in Aut(\mathbb{F}_{q^n})$ such that one among these $m$ elements is a $q^{mK}-1$ power:\begin{eqnarray*}
     C_1&:=&\left( \frac{\overline{\alpha_{2}}}{{\alpha}^{\sigma}_{2}}\right)\left( \frac{\overline{\alpha_{3}}}{{\alpha}^{\sigma}_{3}}\right)^{q^{K}}\cdots\left( \frac{\overline{\alpha_{m}}}{{\alpha}^{\sigma}_{m}}\right)^{q^{(m-2)K}}\left(\frac{\overline{\alpha_{1}}}{{\alpha}^{\sigma}_1}\right)^{q^{(m-1)K}}\\
     C_{\delta}&:=&\left( \frac{\overline{\alpha_{\delta+1}}}{{\alpha}^{\sigma}_{2}}\right)\cdots\left( \frac{\overline{\alpha_{m}}}{{\alpha}^{\sigma}_{m-\delta+1}}\right)^{q^{(m-\delta-1)K}}
     \left(\frac{\overline{\alpha_{1}}}{{\alpha}^{\sigma}_{m-\delta+2}}\right)^{q^{(m-\delta)K}}\cdots
     \left(\frac{\overline{\alpha_{\delta-1}}}{{\alpha}^{\sigma}_m}\right)^{q^{(m-2)K}}
     \left(\frac{\overline{\alpha_{\delta}}}{{\alpha}^{\sigma}_1}\right)^{q^{(m-1)K}}\\
     C_m&:=&\left(\frac{\overline{\alpha_{1}}}{{\alpha}^{\sigma}_{2}}\right)\cdots\left(\frac{\overline{\alpha_{m-1}}}{{\alpha}^{\sigma}_m}\right)^{q^{(m-2)K}}\left(\frac{\overline{\alpha_{m}}}{{\alpha}^{\sigma}_1}\right)^{q^{(m-1)K}},
 \end{eqnarray*}
 with $\delta=2,\dots,m-1$.
 \end{theorem}
 \begin{proof}

Let $\sigma\in Aut(\mathbb{F}_{q^n})$ be fixed. Consider System (\ref{system}) with $I_0=I,J_0=J$. In particular for any $\delta=1,\dots,m-1$ 
 \begin{align*}
a_{(m+\delta,1)}\widetilde{x}_1+\cdots+a_{(m+\delta,m)}\widetilde{x}_m+a_{(m+\delta,m+1)}(\widetilde{x}_1\qi+A_2\widetilde{x}_2\qj)+\cdots+a_{(m+\delta,2m)}(\widetilde{x}_m\qi+A_1\widetilde{x}_1\qj)&+\\-a^{q^I}_{(\delta,1)}\widetilde{x}_1\qi-\cdots-a^{q^I}_{(\delta, m)}\widetilde{x}_m\qi&+\\-a^{q^I}_{(\delta,m+1)}(\widetilde{x}_1^{q^{2I}}+A_2\qi\widetilde{x}_2^{q^{J+I}})-\cdots-a^{q^I}_{(\delta,2m)}(\widetilde{x}_m^{q^{2I}}+A_1\qi\widetilde{x}_1^{q^{J+I}})&+\\-\overline{\alpha_{\delta+1}}(a^{q^J}_{(\delta+1,1)}\widetilde{x}_1\qj+\cdots+a^{q^J}_{(\delta+1,m)}\widetilde{x}_m\qj&+\\+a^{q^J}_{(\delta+1,m+1)}(\widetilde{x}_1^{q^{I+J}}+A_2\qj\widetilde{x}_2^{q^{2J}})+\cdots+a^{q^J}_{(\delta+1,2m)}(\widetilde{x}_m^{q^{I+J}}+A_1\qj\widetilde{x}_1^{q^{2J}}))&=0\tag{a$_\delta$}\label{eq1.1}
 \end{align*}
 and
  \begin{align*}
a_{(2m,1)}\widetilde{x}_1+\cdots+a_{(2m,m)}\widetilde{x}_m+a_{(2m,m+1)}(\widetilde{x}_1\qi+A_2\widetilde{x}_2\qj)+\cdots+a_{(2m,2m)}(\widetilde{x}_m\qi+A_1\widetilde{x}_1\qj)&+\\-a^{q^I}_{(m,1)}\widetilde{x}_1\qi-\cdots-a^{q^I}_{(m,m)}\widetilde{x}_m\qi&+\\-a^{q^I}_{(m,m+1)}(\widetilde{x}_1^{q^{2I}}+A_2\qi\widetilde{x}_2^{q^{J+I}})-\cdots-a^{q^I}_{(m,2m)}(\widetilde{x}_m^{q^{2I}}+A_1\qi\widetilde{x}_1^{q^{J+I}})&+\\-\overline{\alpha_{1}}(a^{q^J}_{(1,1)}\widetilde{x}_1\qj+\cdots+a^{q^J}_{(1,m)}\widetilde{x}_m\qj&+\\+a^{q^J}_{(1,m+1)}(\widetilde{x}_1^{q^{I+J}}+A_2\qj\widetilde{x}_2^{q^{2J}})+\cdots+a^{q^J}_{(1,2m)}(\widetilde{x}_m^{q^{I+J}}+A_1\qj\widetilde{x}_1^{q^{2J}}))&=0.\tag{a$_m$}\label{eq1.2}
 \end{align*}
 If the above equations are satisfied for every $\widetilde{x}_1,\dots,\widetilde{x}_m$ then by checking the coefficients of $\widetilde{x}_1, \dots, \widetilde{x}_m, \widetilde{x}_1^{q^{I+J}}, \dots, \widetilde{x}_m^{q^{I+J}}, \widetilde{x}_m^{q^{2J}} \widetilde{y}^{q^{2J}}, \widetilde{z}^{q^{2J}}$  we obtain
$$a_{(m+\delta,1)}=\cdots=a_{(m+\delta, m)}=a_{(\delta,m+1)}=\cdots=a_{(\delta,2m)}=0.$$
Thus the equations read:
 \begin{align*}
a_{(m+\delta,m+1)}(\widetilde{x}_1\qi+A_2\widetilde{x}_2\qj)+\cdots+a_{(m+\delta,2m)}(\widetilde{x}_m\qi+A_1\widetilde{x}_1\qj)-a^{q^J}_{(\delta,1)}\widetilde{x}_1\qi-\cdots-a^{q^I}_{(\delta, m)}\widetilde{x}_m\qi&+\\-\overline{\alpha_{\delta+1}}(a^{q^J}_{(\delta+1,1)}\widetilde{x}_1\qj+\cdots+a^{q^J}_{(\delta+1,m)}\widetilde{x}_m\qj)=0,
 \end{align*}
  with $\delta=1,\dots,m-1$, and
  \begin{align*}
a_{(2m,m+1)}(\widetilde{x}_1\qi+A_2\widetilde{x}_2\qj)+\cdots+a_{(2m,2m)}(\widetilde{x}_m\qi+A_1\widetilde{x}_1\qj)-a^{q^J}_{(m,1)}\widetilde{x}_1\qi-\cdots-a^{q^J}_{(m, m)}\widetilde{x}_m\qi&+\\-\overline{\alpha_{1}}(a^{q^J}_{(1,1)}\widetilde{x}_1\qj+\cdots+a^{q^J}_{(1,m)}\widetilde{x}_m\qj)=0.
 \end{align*}

 So
 \begin{align*}
   \widetilde{x}_1\qi(a_{(m+\delta,m+1)}-a^{q^I}_{(\delta,1)})+\widetilde{x}_2\qi(a_{(m+\delta,m+2)}-a^{q^I}_{(\delta,2)})+\dots+ \widetilde{x}_m\qi(a_{(m+\delta,2m)}-a^{q^I}_{(\delta,m)})+\\+\widetilde{x}_1\qj(a_{(m+\delta,2m)}A_1-\overline{\alpha_{\delta+1}}a^{q^J}_{(\delta+1,1)})+\widetilde{x}_2\qj(a_{(m+\delta,m+1)}A_2-\overline{\alpha_{\delta+1}}a^{q^J}_{(\delta+1,2)})+\\+\dots+\widetilde{x}_m\qj(a_{(m+\delta,2m-1)}A_m-\overline{\alpha_{\delta+1}}a^{q^J}_{(\delta+1,m)})=0,
 \end{align*}
   with $\delta=1,\dots,m-1$, and
  \begin{align*}
   \widetilde{x}_1\qi(a_{(2m,m+1)}-a^{q^I}_{(m,1)})+\widetilde{x}_2\qi(a_{(2m,m+2)}-a^{q^I}_{(m,2)})+\dots+ \widetilde{x}_m\qi(a_{(2m,2m)}-a^{q^I}_{(m,m)})+\\+\widetilde{x}_1\qj(a_{(2m,2m)}A_1-\overline{\alpha_{1}}a^{q^J}_{(1,1)})+\widetilde{x}_2\qj(a_{(2m,m+1)}A_2-\overline{\alpha_{1}}a^{q^J}_{(1,2)})+\\+\dots+\widetilde{x}_m\qj(a_{(2m,2m-1)}A_m-\overline{\alpha_{1}}a^{q^J}_{(1,m)})=0.
 \end{align*}
 These polynomials are identically zero if and only if
 $$
 \begin{array}{cc}
    a_{(m+\delta,m+\tau)}=a^{q^I}_{(\delta,\tau)},&a_{(m+\delta,2m)}=a^{q^I}_{(\delta,m)},\vspace{2mm}\\ a_{(m+\delta,m+\tau)}=\frac{\overline{\alpha_{\delta+1}}}{A_{\tau+1}}a^{q^J}_{(\delta+1,\tau+1)},&a_{(m+\delta,2m)}=\frac{\overline{\alpha_{\delta+1}}}{A_1}a^{q^J}_{(\delta+1,1)},\vspace{2mm}\\
a_{(2m,m+\tau)}=a^{q^I}_{(m,\tau)},&a_{(2m,2m)}=a^{q^I}_{(m,m)},\vspace{2mm}\\ 
a_{(2m,m+\tau)}=\frac{\overline{\alpha_{1}}}{A_{\tau+1}}a^{q^J}_{(1,\tau+1)},&a_{(2m,2m)}=\frac{\overline{\alpha_{1}}}{A_1}a^{q^J}_{(1,1)}
 \end{array}$$
 with $\delta,\tau=1,\dots,m-1$.
 
     Thus $$
     \begin{array}{cc}
a_{(\delta,\tau)}=\left( \frac{\overline{\alpha_{\delta+1}}}{A_{\tau+1}}\right)^{q^{-I}} a^{q^K}_{(\delta+1,\tau+1)},& a_{(\delta,m)}= \left(\frac{\overline{\alpha_{\delta+1}}}{A_1}\right)^{q^{-I}}a^{q^K}_{(\delta+1,1)},\vspace{2mm}\\
a_{(m,\tau)}=\left(\frac{\overline{\alpha_{1}}}{A_{\tau+1}}\right)^{q^{-I}}a^{q^K}_{(1,\tau+1)},&a_{(m,m)}=\left(\frac{\overline{\alpha_{1}}}{A_1}\right)^{q^{-I}}a^{q^K}_{(1,1)}
        
     \end{array}$$  with $\delta,\tau=1,\dots,m-1$.
     
 Thus 
 \begin{eqnarray*}
     a_{(1,1)}&=&\left( \frac{\overline{\alpha_{2}}}{A_{2}}\right)^{q^{-I}} a^{q^K}_{(2,2)}\\&=&\left( \frac{\overline{\alpha_{2}}}{A_{2}}\right)^{q^{-I}}\left( \frac{\overline{\alpha_{3}}}{A_{3}}\right)^{q^{K-I}} a^{q^{2K}}_{(3,3)}\\&=&
     \left( \frac{\overline{\alpha_{2}}}{A_{2}}\right)^{q^{-I}}\left( \frac{\overline{\alpha_{3}}}{A_{3}}\right)^{q^{K-I}}\cdots\left( \frac{\overline{\alpha_{m}}}{A_{m}}\right)^{q^{(m-2)K-I}}a^{q^{(m-1)K}}_{(m,m)}\\&=&\left( \frac{\overline{\alpha_{2}}}{A_{2}}\right)^{q^{-I}}\left( \frac{\overline{\alpha_{3}}}{A_{3}}\right)^{q^{K-I}}\cdots\left( \frac{\overline{\alpha_{m}}}{A_{m}}\right)^{q^{(m-2)K-I}}\left(\frac{\overline{\alpha_{1}}}{A_1}\right)^{q^{(m-1)K-I}}a^{q^{mk}}_{(1,1)},
     \end{eqnarray*}
     \begin{eqnarray*}  
     a_{(\delta,1)}&=&\left( \frac{\overline{\alpha_{\delta+1}}}{A_{2}}\right)^{q^{-I}} a^{q^K}_{(\delta+1,2)}\\
     &=&\left( \frac{\overline{\alpha_{\delta+1}}}{A_{2}}\right)^{q^{-I}}\cdots\left( \frac{\overline{\alpha_{m}}}{A_{m-\delta+1}}\right)^{q^{(m-\delta-1)K-I}}a^{q^{(m-\delta)K}}_{(m,m-\delta+1)}\\
     &=&\left( \frac{\overline{\alpha_{\delta+1}}}{A_{2}}\right)^{q^{-I}}\cdots\left( \frac{\overline{\alpha_{m}}}{A_{m-\delta+1}}\right)^{q^{(m-\delta-1)K-I}}
     \left(\frac{\overline{\alpha_{1}}}{A_{m-\delta+2}}\right)^{q^{(m-\delta)K-I}}a^{q^{(m-\delta+1)K}}_{(1,m-\delta+2)}\\
     &=&\left( \frac{\overline{\alpha_{\delta+1}}}{A_{2}}\right)^{q^{-I}}\cdots
     \left(\frac{\overline{\alpha_{1}}}{A_{m-\delta+2}}\right)^{q^{(m-\delta)K-I}}\cdots
     \left(\frac{\overline{\alpha_{\delta-1}}}{A_m}\right)^{q^{(m-2)K-I}}a^{q^{(m-1)K}}_{\delta-1,m}\\
     &=&\left( \frac{\overline{\alpha_{\delta+1}}}{A_{2}}\right)^{q^{-I}}\cdots
     \left(\frac{\overline{\alpha_{\delta-1}}}{A_m}\right)^{q^{(m-2)K-I}}
     \left(\frac{\overline{\alpha_{\delta}}}{A_1}\right)^{q^{(m-1)K-I}}a^{q^{mK}}_{(\delta,1)},\\ 
     a_{(m,1)}&=&\left(\frac{\overline{\alpha_{1}}}{A_{2}}\right)^{q^{-I}}a^{q^K}_{(1,2)}\\
     &=&\left(\frac{\overline{\alpha_{1}}}{A_{2}}\right)^{q^{-I}}\cdots\left(\frac{\overline{\alpha_{m-1}}}{A_m}\right)^{q^{(m-2)K-I}}a^{q^{(m-1)K}}_{(m-1,m)}\\
     &=&\left(\frac{\overline{\alpha_{1}}}{A_{2}}\right)^{q^{-I}}\cdots\left(\frac{\overline{\alpha_{m-1}}}{A_m}\right)^{q^{(m-2)K-I}}\left(\frac{\overline{\alpha_{m}}}{A_1}\right)^{q^{(m-1)K-I}}a^{q^{mK}}_{(m,1)}.
     \end{eqnarray*}    
     
     This shows that $a_{(\delta,1)}=C_{\delta}^{q^{-I}}a_{(\delta,1)}^{q^{mK}}$, $\delta=1,\dots,m$.

     Note that if all $C_{\delta}$ are not $(q^{mK}-1)$-power, then $a_{(\delta,1)}=0$ for all $\delta=1,\dots,m$ thus $M\notin \GL(2m,{q^n})$.
     
     Conversely, if there exists $\overline{\delta}\in [1,\dots,m]$ such that $C_{\overline{\delta}}$ is a $(q^{mK}-1)$-power, there exists a nonzero $a_{(\delta,1)}$ satisfying $a_{(\overline{\delta},1)}=C_{\overline{\delta}}^{q^{-I}}a_{(\overline{\delta},1)}^{q^{mK}}$ and, by considering $a_{(\delta,1)}=0$ for all $\delta\neq\overline{\delta}$, we can construct a non singular $M\in GL(2m,{q^n})$ and the two sets are $\Gamma \mathrm{L}$-equivalent.  
 \end{proof}

This necessary and sufficient condition allows us to obtain a lower bound on the number of inequivalent scattered sequences.

\begin{prop}
   Given $I,J$ such that $J>I,J>\frac{n}{2}$, if $n=mn'$ and $m\not|\, K$ then there are at least \begin{equation*}
    \frac{1-c^{q,m}}{mnh}\cdot(q^{m\gcd(n',K)})
    \end{equation*} inequivalent scattered sets $\us$.
\end{prop}

\begin{proof}

Fix  $(\overline{\alpha_1},\dots,\overline{\alpha_m})\in\mathbb{F}^m_{q^n}$ and consider all the  $(\alpha_1, \dots, \alpha_m)\in\mathbb{F}^m_{q^n}$ such that the corresponding sets are equivalent. For given $\alpha_2,\dots,\alpha_m$, the function $$\alpha_1 \longmapsto C_1=\left( \frac{\overline{\alpha_{2}}}{{\alpha}^{\sigma}_{2}}\right)\left( \frac{\overline{\alpha_{3}}}{{\alpha}^{\sigma}_{3}}\right)^{q^{K}}\cdots\left( \frac{\overline{\alpha_{m}}}{{\alpha}^{\sigma}_{m}}\right)^{q^{(m-2)K}}\left(\frac{\overline{\alpha_{1}}}{{\alpha}^{\sigma}_1}\right)^{q^{(m-1)K}} $$ is a permutation of $\field$. Since $\alpha_2,\dots,\alpha_m$ can vary in $(q^n-1)^{m-1}$ ways, $C_1$ is a $(q^{mk}-1)$-power for $(q^n-1)^m/{\gcd(q^{mK}-1,q^n-1)}$ $m
$-uples $(\alpha_1, \dots, \alpha_m)$.

An equivalence with $\sigma\neq id$ corresponds to an equivalence with $(\alpha_1^{\sigma}, \dots, \alpha_1^{\sigma})$. Via the condition on $C_1$, there are  at most $nh(q^n-1)^m/({\gcd(q^{mK}-1,q^n-1)})$ sets $\us$ equivalent to $U_{\bold{\overline{A}}}^{I,J}$.

Arguing analogously for  $C_2,\dots,C_m$, we obtain that  $$\frac{mnh(q^n-1)^m}{\gcd(q^{mK}-1,q^n-1)}$$
is an upper bound for the number of sets $\us$ equivalent to a fixed $U_{\bold{\overline{A}}}^{I,J}$.

Using the lower bound on the number of distinct instances of $(\alpha_1, \dots, \alpha_m)$ giving rise to a scattered and indecomposable set $\us$, remarking that $n=mn'$, we can obtain a lower bound on the number of inequivalent scattered sequences
\begin{equation*}
    (q^n-1)^{m-1}\left((q^n-1)-\frac{q^n-1}{\gcd(q^n-1,C_{K,m})}-\sum_{j=1}^{\lceil \frac{m-1}{2}\rceil}(q^n-1)\frac{q^{\gcd(mn',j)}-1}{q^{m\gcd(n',K)}-1}\right)\cdot\frac{q^{m\gcd(n',K)}-1}{mnh(q^n-1)^m}
\end{equation*}
\begin{equation*}
    \geq\frac{1-c^{q,m}}{mnh}\cdot(q^{m\gcd(n',K)}).
\end{equation*}
    
\end{proof}
 \begin{comment}
        \begin{ex}
        With $q=8,m=4,I=1,J=3,n=8$ we have at least $82004$ inequivalent examples with these parameters. 
    \end{ex}
    \end{comment}

%%%%%%%%%%%%%%%%%%%%%%%%%%%%%%%%%%%%%%%%%%%%%%%%%%%
 \section*{ Acknowledgement } This work was  supported by the
Research Project of MIUR (Italian Office for University and
Research) ``Strutture Geometriche, Combinatoria e loro Applicazioni''.

%%%%%%%%%%%%%%%%%%%%%%%%%%%%%%%%%%%%%%%%%%%%%%%%%%%

\section*{Conflict of interest}
 On behalf of all authors, the corresponding author states that there is no conflict of interest. 
\bibliographystyle{abbrv}

\end{document}